\documentclass[oneside,11pt]{amsart}

\usepackage{mathtools}
\usepackage{amssymb,amsmath,amsthm}
\usepackage[foot]{amsaddr}
\usepackage[mathscr]{eucal}
\usepackage{tikz-cd}
\usepackage{hyperref}
\usepackage{verbatim}
\usepackage{graphicx}
\allowdisplaybreaks

\usepackage[textwidth=17cm,top=2cm,bottom=2cm]{geometry}

\newtheorem{thm}{Theorem}[section]
\newtheorem{prop}[thm]{Proposition}
\newtheorem{lem}[thm]{Lemma}
\newtheorem{cor}[thm]{Corollary}

\theoremstyle{definition}

\newtheorem{ex}[thm]{Example}

\theoremstyle{remark}

\newtheorem{rmk}[thm]{Remark}
\newtheorem{note}[thm]{Notation}

\numberwithin{equation}{section}

\newcommand{\bA}{\mathbb{A}}

\newcommand{\bG}{\mathbb{G}}

\newcommand{\bP}{\mathbb{P}}

\newcommand{\bV}{\mathbb{V}}

\newcommand{\bZ}{\mathbb{Z}}

\newcommand{\cI}{\mathcal{I}}

\newcommand{\cL}{\mathcal{L}}
\newcommand{\cM}{\mathcal{M}}

\newcommand{\cO}{\mathcal{O}}

\newcommand{\fm}{\mathfrak{m}}

\DeclareMathOperator{\Aut}{Aut}
\DeclareMathOperator{\Id}{Id}
\DeclareMathOperator{\Ker}{Ker}

\DeclareMathOperator{\Spec}{Spec}

\DeclareMathOperator{\Sym}{Sym}
\DeclareMathOperator{\PGL}{PGL}
\DeclareMathOperator{\GL}{GL}
\DeclareMathOperator{\pr}{pr}
\DeclareMathOperator{\Hilb}{Hilb}
\DeclareMathOperator{\Pic}{Pic}
\DeclareMathOperator{\Gr}{Gr}

\DeclareMathOperator{\Pf}{Pf}
\DeclareMathOperator{\Bl}{Bl}
\DeclareMathOperator{\Bs}{Bs}
\DeclareMathOperator{\Fit}{Fit}
\DeclareMathOperator{\Stab}{Stab}

\newcommand{\pencil}[1]{\langle {#1} \rangle}

\newcommand{\nocontentsline}[3]{}
\let\origcontentsline\addcontentsline
\newcommand\stoptoc{\let\addcontentsline\nocontentsline}
\newcommand\resumetoc{\let\addcontentsline\origcontentsline}

\begin{document}


\title{Automorphisms of odd dimensional $(2,2)$-complete intersections in characteristic $2$}




\author{Yang Zhang}
\address{Institut de Math\'ematiques, \'Ecole polytechnique f\'ed\'erale de Lausanne}
\email{yang.zhang@epfl.ch}



\begin{abstract}
We compute the automorphism scheme of a generic odd dimensional $(2,2)$-complete intersection in characteristic $2$.
\end{abstract}


\maketitle


\tableofcontents


\section{Introduction}
A main obstruction to the existence of a fine moduli space of a class of varieties is the non-triviality of their automorphisms. 
Therefore it becomes necessary to study the automorphisms of a class of varieties in order to understand its moduli theory. 
For complete intersections in projective space, which is often considered as one of the most well-understood classes of varieties, there are already plenty of results on their automorphisms.
For example, in \cite{Benoist2011SparationEP}, the following theorem is proved:

\begin{thm}[{\cite[Th\'eor\`eme 1.6]{Benoist2011SparationEP}}]\label{thm:DM stack}
	Fix integers $N\geq 2, c\geq 1, d_1\geq \dots\geq d_c\geq 2$. 
	The Stack $\cM_{N,c,(d_i)_{1\leq i\leq c}}$ classifying smooth $(d_1,\dots, d_c)$-complete intersections in $\bP^N$ polarized by $\cO(1)$ is Deligne-Mumford over $\Spec \bZ$ except for the following cases:
	\begin{enumerate}
		\item If $c=1$ and $d_1=2$.
		\item If $N=2$, $c=1$ and $d_1=3$. In this case $\cM_{2,1,(3)}$ is Deligne-Mumford over $\Spec \bZ[\frac 1 3]$.
		\item If $N$ is odd, $c=2$ and $d_1=d_2=2$. In this case $\cM_{N,2,(2,2)}$ is Deligne-Mumford over $\Spec \bZ[\frac 1 2]$.
	\end{enumerate}
\end{thm}

As a Corollary, we obtain that if $X$ is a smooth complete intersection over an algebraically closed field $k$, then the connected component of identity $\Aut^0_{X,\cO_X(1)}$ in the polarized automorphism scheme $\Aut_{X,\cO_X(1)}$ is trivial except for the following cases:
\begin{enumerate}
	\item $X$ is a quadratic hypersurface.
	\item $X$ is a cubic curve in $\bP^2_k$ in characteristic $3$.
	\item $X$ is a $(2,2)$-complete intersection of odd dimension in characteristic $2$.
\end{enumerate}
The study of the first case degenerates to the study of quadratic forms, and the study of the second case is involved in the study of elliptic curves. 
So it is natural to ask whether we can compute $\Aut_{X,\cO_X(1)}$ for $X$ an odd dimensional $(2,2)$-complete intersection in characteristic $2$.

From the point of view of quadratic forms, it is also interesting to study the geometry of pairs of quadrics. 
There is plenty of literature summarizing the results for pairs of quadratic forms in characteristic $\neq 2$, and \cite{dolgachev2018regular} gives a detailed study for pairs of quadratic forms with an odd number of variables in characteristic $2$. 
The only remaining case is pairs with an even number of variables in characteristic $2$.
A normal form for such a pair being generic in moduli is given in \cite{boshle1990pencils}, and \cite[Proposition 3.8]{Benoist2011SparationEP} gives a lower bound of the dimension of the space of infinitesimal automorphisms, but no further result about automorphisms is known.

In this paper, we compute the automorphism scheme of a generic smooth odd dimensional $(2,2)$-complete intersection whose dimension is at least $3$ in projective space over an algebraically closed field of characteristic $2$. 
To be more precise, we classify the automorphism scheme for those $(2,2)$-complete intersections whose Pfaffian has distinct roots. 
The paper is organized as follows.
In Section 2, we give the necessary preliminaries needed for the study of automorphisms. We recall that the automorphism scheme of such a complete intersection is finite. As a consequence, the automorphism scheme splits into an infinitesimal part and an \'etale part:
\[
	\Aut_X \cong \Aut^0_X \rtimes \pi_0(\Aut_X).
\]
We define also the Pfaffian polynomial corresponding to such a complete intersection. 
In Section 3, we compute the identity component of the automorphism scheme:
\begin{thm}[cf. Theorem \ref{thm:structure of Aut^0}]
	Assume $M\geq 2$ and $\mathrm{char}\ k=2$, then the identity component $\Aut^0_X$ of the automorphism scheme of a smooth $(2,2)$-complete intersection $X$ in $\bP^{2M+1}_k$ whose Pfaffian has distinct roots is isomorphic to $\mu_2^{M}$.
\end{thm}
We remark again, as pointed out after Theorem \ref{thm:DM stack}, this is generic in the only case where a non-hypersurface complete intersection has a non-trivial identity component of automorphism schemes. 
In particular, the identity component is trivial even for $(2,2)$-complete intersections that are not in characteristic $2$, as well as those in characteristic $2$ of an even dimension (See also \cite{dolgachev2018regular}).
In Section 4, we study the structure of the \'etale part of the automorphism scheme. The main result can be summarized as follows: 
\begin{thm}[cf. Theorem \ref{thm: exact sequence for pi_0} and Corollary \ref{cor: pi_0 for a generic X}]
	Assume as above $M\geq 2$, $\mathrm{char}\ k=2$, $X$ is smooth and the Pfaffian of $X$ has distinct roots.
	There is an exact sequence of groups:
	\[\begin{tikzcd}
		1 & {(\bZ/2\bZ)^{M+1}} & {\pi_0(\Aut_X)} & {\Aut(\bP_k^1;\varphi(Z)),}
		\arrow[from=1-1, to=1-2]
		\arrow[from=1-2, to=1-3]
		\arrow[from=1-3, to=1-4]
	\end{tikzcd}\]
	where $Z$ is the singular locus of the morphism $\varphi$ defined by the pencil $\pencil{f,g}$ (c.f. Proposition \ref{prop:fitting ideal of pencil}), and $\Aut(\bP_k^1;\varphi(Z))$ is the equivariant automorphism group of $\varphi(Z)$ in $\bP_k^1$. Moreover, ${\Aut(\bP_k^1;\varphi(Z))}$ is trivial for a generic $X$ in the moduli of smooth $(2,2)$-complete intersections, so $\pi_0(\Aut_X)\cong (\bZ/2\bZ)^{M+1}$ in this case.
\end{thm}
We give also a full criterion on the liftability of elements in $\Aut(\bP_k^1; \varphi(Z))$ to $\pi_0(\Aut_X)$, c.f. Theorem \ref{thm:characterization of image of phi_*}, and construct liftable and non-liftable examples, c.f. Example \ref{ex: liftable element in PGL_2} and Example \ref{ex: non-liftable element in PGL_2}.

\stoptoc
\subsection*{Conventions and Notations}
We work over an algebraically closed field $k$ of characteristic $2$, unless otherwise mentioned. 
A variety means a integral separated $k$-scheme of finite type. 
The notation $\pencil{f,g}$ means the vector subspace spanned by $f$ and $g$. 

\subsection*{Acknowledgements}
The author is sincerely grateful to Gebhard Martin for proposing this research project and providing detailed suggestions throughout his master's studies in Bonn. The author is also thankful to Raymond Cheng for many fruitful discussions on this topic.

\resumetoc
\section{Preliminaries}
\subsection{Finiteness of $\Aut_X$}
Let $X$ be a smooth complete intersection in $\bP_k^{N}$ over an arbitrary field $k$ such that $X$ is not a quadric hypersurface. 
We consider the polarized automorphism scheme $\Aut_{X,\cO_X(1)}$, which is defined as the stabilizer of the class $[\cO_X(1)]\in \Pic_X(k)$ under the natural action of $\Aut_X$ on $\Pic_X$. 
In particular, $\Aut_{X,\cO_X(1)}$ is a closed subgroup scheme of $\Aut_X$.
\begin{prop}[{\cite[Th\'eor\`eme 1.7]{Benoist2011SparationEP}}]
	The stack parameterizing smooth complete intersection in $\bP_k^N$ polarized by $\cO(1)$ is separated except for the component for quadric hypersurfaces. In particular, the polarized automorphism scheme $\Aut_{X,\cO_X(1)}$ is proper over $k$.
\end{prop}

\begin{prop}[{\cite[Remark 2.15(ii)]{brion2018notes}}]\label{prop:polarized Aut is algebraic}
	There is a closed embedding $\Aut_{X,\cO_X(1)} \hookrightarrow \PGL_{N+1}$ with respect to which $X\hookrightarrow \bP_k^N$ is equivariant.
	In particular, $\Aut_{X,\cO_X(1)}$ is an algebraic group.
\end{prop}

\begin{cor}\label{cor: polarized Aut is finite}
	The polarized automorphism scheme $\Aut_{X,\cO_X(1)}$ is finite.
\end{cor}

We can show that $\Aut_{X,\cO_X(1)}=\Aut_X$ if $X$ is a smooth complete intersection in $\bP^N_k$ whose dimension is at least $3$ and whose canonical bundle $\omega_X$ is non-trivial. 
To prove the statement, we need the following refinement of the Grothendieck-Lefschetz theorem:

\begin{lem}\label{lem:grothendieck-lefschetz}
	Let $X$ be a smooth complete intersection in $\bP^N_k$ over an algebraically closed field $k$ such that $\dim X\geq 3$, then we have $\Pic_X= \underline \bZ$.
\end{lem}
\begin{proof}
	It suffices to prove that for any $k$-scheme $T$, we have $\Pic(X_T) \cong \bZ \cO_{X_T}(1) \times \Pic(T)$.
	The case where $T=\Spec k$ follows from \cite[Expos\'e XII, Corollaire 3.7]{grothendieck2005cohomologie}. 
	Moreover, $H^1(X,\cO_X)=0$ and it is proved in \cite[Chapter III, Exercise 12.6]{hartshorne1977algebraic} that in this case $\Pic(X\times T)\cong \Pic(X)\times \Pic(T)$ for all $T$. 
\end{proof}

\begin{prop}\label{prop:Aut and polarized Aut agree}
	Let $X$ be a smooth complete intersection in $\bP^N_k$ of dimension at least $3$ and whose canonical line bundle is non-trivial. 
	Then $\Aut_X=\Aut_{X,\cO_X(1)}$.
\end{prop}
\begin{proof}
	Pick $l$ such that $\cO_X(l)$ is the canonical line bundle of $X$. 
	Then $\cO_{X_T}(l)$ is the relative canonical line bundle of $X_T$ over $T$, which is preserved by every automorphism of $X_T$ over $T$. 
	Therefore, $\cO_{X_T}(1)$ is pulled back along an automorphism to $\cO_{X_T}(1)\otimes \cL$, where $\cL$ is an $l$-torsion element in $\Pic_X(T)$. 
	Since $\Pic_X(T)$ is torsion-free by Lemma \ref{lem:grothendieck-lefschetz}, we get $\cL \cong \pr_T* \cM$ for some $\cM \in \Pic(T)$ and the class $[\cO_{X_T}(1)] \in \Pic_X(T)$ is preserved.
\end{proof}

\begin{cor}\label{cor:Aut is finite}
	Let $X$ be a smooth complete intersection in $\bP^N_k$ such that 
	\begin{enumerate}	
		\item $\dim X\geq 3$,
		\item $\omega_X\ncong \cO_X$,
		\item $X$ is not a quadric hypersurface.
	\end{enumerate}
	Then the automorphism scheme $\Aut_X$ is finite.
\end{cor}
\begin{proof}
	Combining Corollary \ref{cor: polarized Aut is finite} and Proposition \ref{prop:Aut and polarized Aut agree}.
\end{proof}

Due to Proposition \ref{prop:polarized Aut is algebraic} and Proposition \ref{prop:Aut and polarized Aut agree}, we consider from now on $\Aut_X$ as a subgroup scheme of $\PGL_{N+1}$.

\subsection{Pfaffian of pencils of quadratic forms in characteristic $2$}

Let $k$ be a field of characteristic $2$ and $q\in H^0(\bP^{2M+1}_k,\cO(2))$ be a quadratic form on $k$ with an even number of variables. 
We associate to $q$ an alternating (hence also symmetric) bilinear form $A_q(x,y):= q(x+y)-q(x)-q(y)$. 
Under a suitable choice of basis of $k^{2M+2}$, the form $A_q$ admits an orthogonal decomposition $A_q\cong \pencil{0}^{2l} \oplus U^{M-l+1}$, where $\pencil{0}$ is the zero bilinear form on $k$, and $U$ is the hyperbolic bilinear form on $k^2$ (cf. \cite[Proposition 7.31]{elman2008algebraic}). 
In particular, $\det(A_q)$ is always a square ,since the conjugation by $S$ yields $\det(SA_qS^T)=\det (S)^2\cdot \det(A_q)$. 
The Pfaffian of $q$ is defined as $\Pf(q):=\sqrt{\det(A_q)}$ by choosing a matrix representation of $A_q$, which is well defined up to a scalar in $k^\times/(k^\times)^2$.

More generally, given a pencil of quadratic forms $\pencil{f,g}$, one can define
\[
	\Pf_{\pencil{f,g}}(t):=\Pf(f-tg),
\]
which is a polynomial in $t$, well-defined up to a scalar in $k^\times/(k^\times)^2$. 
If $X=V(f,g)$ is the corresponding $(2,2)$-complete intersection, we also call $\Pf_{\pencil{f,g}}$ the Pfaffian of $X$. 
In \cite{boshle1990pencils}, a normal form of pencils of quadratic forms in characteristic $2$ whose Pfaffian has no multiple roots is given.

\begin{thm}[{\cite[Main Proposition]{boshle1990pencils}}]
	For a suitable choice of coordinates $[x_0:\dots,x_M:y_0:\dots:y_M]$ in $\bP^{2M+1}_k$, a pencil in $H^0(\bP_k^{2M+1}, \cO(2))$ whose Pfaffian has distinct roots can be written as
	\[
		f= \displaystyle\sum_{i=0}^M x_iy_i,\ 
		g= \displaystyle\sum_{i=0}^M a_ix_iy_i + c_ix_i^2 + d_iy_i^2,
	\]
	where $a_i$ are the roots of the Pfaffian, and $V(f,g)$ is non-singular iff $\prod_{i=0}^M c_id_i \neq 0$.
\end{thm}

The case where $\Pf_{\pencil{f,g}}$ has distinct roots is generic, because geometrically it corresponds to the fact that the pencil, considered as a rational curve in the parametrizing space $\bP(H^0(\bP_k^{2M+1}, \cO(2)))$, intersects with the zero locus of the Pfaffian transversally.
To simplify the notations for later, we rescale the coordinates $x_i':=\sqrt{c_i} x_i, y_i':=\sqrt{d_i} y_i$ and obtain the following statement.

\begin{cor}\label{cor:normal form of pencil}
	Let $k$ be an algebraically closed field, and let $\pencil{f,g}$ be a pencil in $H^0(\bP^{2M+1}_k,\cO(2))$. 
	Assume that $V(f,g)$ is smooth and whose Pfaffian has distinct roots, then one may choose coordinates $[x_0:\dots,x_M:y_0:\dots:y_M]$ in $\bP^{2M+1}_k$ such that 
	\[
		f= \displaystyle\sum_{i=0}^M a_ix_iy_i,\ 
		g= \displaystyle\sum_{i=0}^M b_ix_iy_i+x_i^2+y_i^2,
 	\]
	where $a_i\neq 0$ for all $i$, and $\det \left[\begin{smallmatrix}
		a_i & a_j\\
		b_i & b_j
	\end{smallmatrix}\right] \neq 0$ for all $i\neq j$.
\end{cor}

\subsection{The singular locus of a pencil}
Let $k$ be a field. 
Given a linear system $L\subset H^0(\bP^N_k,\cO(d))$, we write $\varphi_L$ for the induced morphism $\bP^N_k\backslash \Bs(L) \rightarrow \bP^S_k$ and $U:=\bP^N_k\backslash \Bs(L)$ for the locus where $\varphi_L$ is defined. 
The cotangent sequence
\[\begin{tikzcd}
	{\varphi_L^* \Omega^1_{\bP^S_k}} & {\Omega^1_U} & {\Omega^1_{\varphi_L}} & 0
	\arrow[from=1-1, to=1-2]
	\arrow[from=1-2, to=1-3]
	\arrow[from=1-3, to=1-4]
\end{tikzcd}\]
gives a locally free resolution of $\Omega^1_{\varphi_L}$. 
Locally ${\varphi_L^* \Omega^1_{\bP^S_k}} \rightarrow \Omega^1_U$ is given by an $(N\times S)$-matrix $A$. 
The $(N-S)$-th fitting ideal $\Fit_{N-S}(\Omega^1_{\varphi_L})$ of $\Omega^1_{\varphi_L}$ is defined as the ideal locally generated by the $(S\times S)$-minors in $A$. 
It is easy to check that this construction agrees on overlaps and therefore defines a global ideal sheaf. 
Let $Z:=V(\Fit_{N-S}(\Omega^1_{\varphi_L}))$ be the closed subscheme cut out by the $(N-S)$-th fitting ideal.

\begin{prop}[{\cite[Lemma 0C3K]{stacks-project}}]\label{prop:locus of fitting ideal is singular locus}
	The support of $Z$ coincides with the points in $U$ where $\varphi_L$ is not smooth.
\end{prop}

The advantage of defining the singular locus of a morphism using the fitting ideal is that this formation commutes with arbitrary base change.

\begin{prop}[{\cite[Lemma 0C3I]{stacks-project}}]\label{prop:singular locus commutes with base change}
	The formation of $Z$ commutes with base change, i.e. given a morphism of schemes $T\rightarrow k$, the pullback $Z_T$ is also the closed subscheme cut out by the $(N-S)$-th fitting ideal of $\Omega^1_{U_T/T}$.
\end{prop}

Now we restrict ourselfs to pencils of quadrics in characteristic $2$. 
Let $L=\pencil{f,g}\subset H^0(\bP_k^{2M+1}, \cO(2))$ be a pencil of quadrics in the normal form given by Corollary \ref{cor:normal form of pencil}:
	\[
		f= \displaystyle\sum_{i=0}^M a_ix_iy_i,\ 
		g= \displaystyle\sum_{i=0}^M b_ix_iy_i+x_i^2+y_i^2,
 	\]
	where $a_i\neq 0$ for all $i$, and $\det \left[\begin{smallmatrix}
		a_i & a_j\\
		b_i & b_j
	\end{smallmatrix}\right] \neq 0$ for all $i\neq j$. 
As a direct corollary of Proposition \ref{prop:singular locus commutes with base change}, we get

\begin{cor}\label{cor:Aut preserves singular locus}
	The action of $\Aut_X \subset \PGL_{2M+2}$ on $\bP_k^{2M+1}$ sends $Z=V(\Fit(\Omega^1_{\varphi_L}))$ to itself.
\end{cor}

The next proposition computes the fitting ideal of $L$ and the corresponding singular locus $Z$ explicitly.

\begin{prop}\label{prop:fitting ideal of pencil}
	The $2M$-th fitting ideal of $\Omega^1_{\varphi_L}$ is $((b_0f+a_0g)y_0,(b_0f+a_0g)x_0, \dots ,(b_Mf+a_Mg)y_M,(b_Mf+a_Mg)x_M)$, and the singular locus $Z$ can be written as a disjoint union $\coprod_{i=0}^M Z_i$, where $Z_i=V(x_0,y_0,\dots, \hat x_i, \hat y_i, \dots, x_n, y_n,  (x_i+y_i)^2)$.
\end{prop}
\begin{proof}
	Let $[s:t]$ be the coordinates of $\bP^1_k$, we may work locally on $D(x_0)$ on the domain and on $D(s)$ on the base. 
	The rest cases will then follow from symmetry. 
	We denote $\bar f=f(1,x_1,\dots, x_M,y_0,\dots, y_M), \bar g=g(1,x_1,\dots, x_M,y_0,\dots, y_M)$. 
	The preimage of $D(s)$ under $\varphi_L$ is $D(f)$, on which one has the following identification:
	\[\begin{tikzcd}
		{\varphi_L^* \Omega^1_{D(s)}} & {\Omega^1_{D(x_0)\cap D(f)}} & {\Omega^1_{\varphi_L}} & 0 \\
		{\cO_{D(x_0)\cap D(f)}\cdot d(\frac g f)} & {\bigoplus_{i=1}^M\cO_{D(x_0)\cap D(f)}\cdot dx_i \oplus \bigoplus_{i=0}^M\cO_{D(x_0)\cap D(f)}\cdot dy_i} & {\Omega^1_{\varphi_L}} & 0.
		\arrow[from=1-1, to=1-2]
		\arrow["{\rotatebox[origin=c]{270}{$\cong$}}"{description}, draw=none, from=1-1, to=2-1]
		\arrow[from=1-2, to=1-3]
		\arrow["{\rotatebox[origin=c]{270}{$\cong$}}"{description}, draw=none, from=1-2, to=2-2]
		\arrow[from=1-3, to=1-4]
		\arrow["{\rotatebox[origin=c]{270}{$=$}}"{description}, draw=none, from=1-3, to=2-3]
		\arrow[from=2-1, to=2-2]
		\arrow[from=2-2, to=2-3]
		\arrow[from=2-3, to=2-4]
	\end{tikzcd}\]
	Then one simply computes
	\begin{align*}
		d(\frac {\bar g}{\bar f})&=
		\frac{1}{\bar f^2} \left(
		\bar f \cdot \sum_{i=1}^M \left( \frac{\partial \bar g}{\partial x_i} \cdot dx_i \right)
		+ 
		\bar f \cdot \sum_{i=0}^M \left( \frac{\partial \bar g}{\partial y_i} \cdot dy_i \right)
		+
		\bar g \cdot \sum_{i=1}^M \left( \frac{\partial \bar f}{\partial x_i} \cdot dx_i \right)
		+ 
		\bar g \cdot \sum_{i=0}^M \left( \frac{\partial \bar f}{\partial y_i} \cdot dy_i \right)
		\right)\\
		&=\frac{1}{\bar f^2} \left(
		\sum_{i=1}^M (b_i\bar f+a_i\bar g)y_i \cdot dx_i
		+ 
		(b_0\bar f+a_0\bar g) \cdot dy_0
		+
		\sum_{i=1}^M (b_i\bar f+a_i\bar g)x_i \cdot dy_i
		\right)
	\end{align*}
	and the fitting ideal is defined by the coefficient elements of $dx_i$ and $dy_i$.
	
	The computation of $V(\Fit_{2M}(\Omega^1_{\varphi_L}))$ is purely set-theoretical. Note that
	\begin{gather*}
		V\left( \bigcup_{i=0}^M ((b_if+a_ig)y_i, (b_if+a_ig)x_i) \right)= \bigcap_{i=0}^M V((b_if+a_ig)y_i, (b_if+a_ig)x_i))
		= \bigcap_{i=0}^M V(b_if+a_ig) \cup V(x_i,y_i)\\
		= \bigcup_{i=0}^M V(x_0,y_0,\dots, \hat x_i, \hat y_i, \dots, x_n, y_n, b_if+a_ig)
		= \bigcup_{i=0}^M V(x_0,y_0,\dots, \hat x_i, \hat y_i, \dots, x_n, y_n, x_i^2+y_i^2).
	\end{gather*}
	The third equality uses that $V(b_if+a_ig)$ and $V(b_jf+a_jg)$ are disjoint in $U$ as $a_ib_j-a_jb_i\neq 0$ for all $i\neq j$. It is obvoius that $Z_i= V(x_0,y_0,\dots, \hat x_i, \hat y_i, \dots, x_n, y_n, x_i^2+y_i^2)$ are mutually distinct, and the second assertion follows.
\end{proof}

\begin{cor}\label{cor: image of singular locus under phi}
	The scheme theoretic image of $Z_i$ under $\varphi$ in $\bP_k^1$ is the point $[a_i:b_i]$ with reduced structure.
\end{cor}
\begin{proof}
	The scheme theoretic image of $Z_i$ is defined by the homogeneous kernel of
	\begin{gather*}
		k[s,t] \rightarrow k[x_0,\dots,x_M,y_0,\dots,y_M]/(x_0,y_0,\dots, \hat x_i, \hat y_i, \dots, x_n, y_n, x_i^2+y_i^2),\\
		s \mapsto f=a_ix_iy_i, t \mapsto g=b_ix_iy_i.
	\end{gather*}
	It is easy to see that the kernel is $(b_is-a_it)$ corresponding to the point $[a_i:b_i]$.
\end{proof}

\subsection{The lifting group with respect to a line bundle}\label{subsection: lifting group}
Let $X$ be a projective variety and $\cL$ a line bundle on $X$. 
Let $L:=\bA_X(\cL)$ be the corresponding geometric line bundle and denote $\pi: L \rightarrow X$ for the natural projection. 
The bundle $L$ comes naturally along with a $\bG_m$ action, and we define $\Aut^{\bG_m}_{X,\cL}$ as the centralizer of $\bG_m$ in $\Aut_L$.
\begin{prop}[{\cite[Proposition 2.13]{brion2018notes}}]\label{prop: exact sequence of lifting group}
	The functor $\Aut^{\bG_m}_\cL$ is represented by a locally algebraic group, and there is an exact sequence
	\[\begin{tikzcd}
		1 & {\bG_m} & {\Aut^{\bG_m}_{X,\cL}} & {\Aut_{X,\cL}} & 1.
		\arrow[from=1-1, to=1-2]
		\arrow[from=1-2, to=1-3]
		\arrow[from=1-3, to=1-4]
		\arrow[from=1-4, to=1-5]
	\end{tikzcd}\]
	Moreover, there is a section $\Aut_{X,\cL} \rightarrow \Aut^{\bG_m}_{X,\cL}$ as morphism of schemes (but not necessarily as morphism of group schemes).
\end{prop}

\begin{ex}
	For projective space $\bP^n_k$ and $\cO(1)$, we simply have $\Aut^{\bG_m}_{\bP^n_k,\cO(1)}\cong \GL_{n+1}$ and we recover $\PGL_{n+1}\cong \GL_{n+1}/\bG_m$.
\end{ex}

Now let $X=V(f,g)$ be a $(2,2)-$complete intersection in $\bP^n_k$,  then there is a morphism of group schemes $\Aut^{\bG_m}_{X,\cO_X(1)} \rightarrow \Aut^{\bG_m}_{\bP^n_k,\cO_{\bP^n_k}(1)} \cong \GL_{n+1}$ making the following diagram commute:
\[\begin{tikzcd}
	1 & {\bG_m} & {\Aut^{\bG_m}_{X,\cO_X(1)}} & {\Aut_{X,\cO_X(1)}} & 1 \\
	1 & {\bG_m} & {\GL_{n+1}} & {\PGL_{n+1}} & {1.}
	\arrow[from=1-1, to=1-2]
	\arrow[from=1-2, to=1-3]
	\arrow[from=1-2, to=2-2]
	\arrow[from=1-3, to=1-4]
	\arrow[from=1-3, to=2-3]
	\arrow[from=1-4, to=1-5]
	\arrow[from=1-4, to=2-4]
	\arrow[from=2-1, to=2-2]
	\arrow[from=2-2, to=2-3]
	\arrow[from=2-3, to=2-4]
	\arrow[from=2-4, to=2-5]
\end{tikzcd}\]
In particular, as the first and the third vertical arrows are closed embeddings, the morphism $\Aut^{\bG_m}_{X,\cO_X(1)} \rightarrow \GL_{n+1}$ is also a closed embedding.

\subsection{Pushforward of automorphisms}
There is a natural action of $\GL_{2M+2}$ on the complete linear system $H^0(\bP^{2M+1}_k,\cO(2))$ together with its pullback to any base. 
By the last section, $\Aut^{\bG_m}_{X,\cO_X(1)}$ is a closed subgroup of $\GL_{2M+2}$, and it fixes the subsystem $\pencil{f,g}\subset H^0(\bP^{2M+1}_k,\cO(2))$. 
We define a morphism from the group $\Aut^{\bG_m}_{X,\cO_X(1)}$ to $\GL_2$ as follows: 
Given $\sigma\in \Aut^{\bG_m}_{X,\cO_X(1)}(T)$, since $\sigma$ fixes the subsystem $\pencil{f_T,g_T}$ in $H^0(\bP^{2M+1}_T,\cO(2))$, there is a unique element in $\GL_2(T)$ corresponding to the action of $\sigma$ on $\pencil{f_T,g_T}$. 
Moreover, by taking the quotients on both sides by $\bG_m$, we get also a morphism $\Aut_X=\Aut_{X,\cO_X(1)} \rightarrow \PGL_2$. 
We denote the both morphisms with $\varphi_*$ by a slight abuse of notation, since we have the following observation:
Take the morphism $\varphi:\bP^{2M+1}_k\backslash X \rightarrow \bP^1_k$ defined by the linear system $\pencil{f,g}$, then the automorphism $\varphi_*\sigma:\bP^1_T \rightarrow \bP^1_T$ over $T$ fits into the following diagram:
\[\begin{tikzcd}
	{\bP^{2M+1}_T\backslash X_T} & {\bP^{2M+1}_T\backslash X_T} \\
	{\bP^1_T} & {\bP^1_T\ .}
	\arrow["\sigma", from=1-1, to=1-2]
	\arrow["{\varphi_T}"', from=1-1, to=2-1]
	\arrow["{\varphi_T}", from=1-2, to=2-2]
	\arrow["{\varphi_*\sigma}", from=2-1, to=2-2]
\end{tikzcd}\]
\begin{rmk}
	If one takes a basis $x_0,\dots,x_M,y_0,\dots,y_M \in H^0(\bP^{2M+1}_T,\cO(1))$ and also consider the $\Aut^{\bG_m}_{X,\cO_X(1)}$-action on it, then the automorphism $\varphi_*(\sigma)$ on $\pencil{f_T,g_T}$ can be described as
	\begin{align*}
		f_T(x_0,\dots,x_M,y_0,\dots,y_M)&\mapsto f_T(\sigma (x_0,\dots,x_M,y_0,\dots,y_M)),\\
		g_T(x_0,\dots,x_M,y_0,\dots,y_M)&\mapsto g_T(\sigma (x_0,\dots,x_M,y_0,\dots,y_M)).
	\end{align*}
\end{rmk}

The fibration $\varphi_{\pencil{f,g}}$ is therefore equivariant with respect to the defined pushforward $\varphi_*:\Aut^{\bG_m}_{X,\cO_X(1)} \rightarrow \GL_2$. 
Let $Z$ be the singular locus of $\varphi$ defined by the fitting ideal of $\pencil{f,g}$. 
By Lemma \ref{cor:Aut preserves singular locus}, the action of $\Aut_X$ on $\bP^{2M+1}_k\backslash X$ sends $Z$ to itself, hence we also have

\begin{prop}\label{prop: pushforawrd of Aut sends image of singular locus to itself}
	The pushforward $\varphi_*\left(\Aut_{X,\cO_X(1)}^{\bG_m} \right)$ (resp. $\varphi_*(\Aut_X)$) sends $\varphi(Z)$ to itself. 
\end{prop}

Using the pushforward, we may vastly restrict the possible forms of elements in $\Aut_{X,\cO_X(1)}^{\bG_m}$.

\begin{note}\label{not: notation for the scalars}
	Assume that $\pencil{f,g}$ is given by the normal form 
	\[
			f= \displaystyle\sum_{i=0}^M a_ix_iy_i,\ 
			g= \displaystyle\sum_{i=0}^M b_ix_iy_i+x_i^2+y_i^2,
 	\]
	where $a_i\neq 0$ for all $i$, and $\det \left[\begin{smallmatrix}
		a_i & a_j\\
		b_i & b_j
	\end{smallmatrix}\right] \neq 0$ for all $i\neq j$
	by Corollary \ref{cor:normal form of pencil}. 
	So $\varphi(Z)$ are the reduced points $\coprod_i [a_i:b_i]$ by Corollary \ref{cor: image of singular locus under phi}.
	Given $R$ an Artinian local ring and $\zeta\in \varphi_*(\Aut_X) (R)$, there is a permutation $\tau\in S_{M+1}$ such that $\zeta \cdot [a_i:b_i]=[a_{\tau(i)}:b_{\tau(i)}]$ by Proposition \ref{prop: pushforawrd of Aut sends image of singular locus to itself}. 
\end{note}

\begin{prop}\label{prop:action of elements in pi_0 on subspace}
	Let $\sigma\in \Aut_{X,\cO_X(1)}^{\bG_m}(R)$ and use the coordinate $[x_0\dots,x_M,y_0,\dots,y_M]$ for $\bP^{2M+1}_R$. 
	Let $\tau$ be the permutation in $S_{M+1}$ associated to $\varphi_*(\sigma)$ in Notation \ref{not: notation for the scalars}.
	The induced action $\sigma^*$ on $H^0(\bP^{2M+1}_R,\cO(1))$ sends the subspaces $\pencil{x_{\tau(i)},y_{\tau(i)}}$ to $\pencil{x_i,y_i}$ and $\pencil{x_{\tau(i)}+y_{\tau(i)}}$ to $\pencil{x_i+y_i}$ for all $0\leq i\leq M$.
	In particular, $\sigma^*$ has the following form:
	\begin{align*}
		x_{\tau(i)} &\mapsto \alpha_i x_i + \beta_i y_i,\\
		y_{\tau(i)} &\mapsto \beta_i x_i + \alpha_i y_i.
	\end{align*}
\end{prop}
\begin{proof}
	Recall that $Z=\coprod_{i=0}^M Z_i$ with $ Z_i=V\left(x_0,y_0,\dots, \hat x_i, \hat y_i, \dots, x_n, y_n, (x_i+y_i)^2 \right)$ by Proposition \ref{prop:fitting ideal of pencil}.
	Since $\sigma$ sends $Z_R:=Z\times R$ to $Z_R$ by Corollary \ref{cor:Aut preserves singular locus} and $\varphi_*(\sigma)$ sends $[a_i:b_i]$ to $[a_{\tau(i)}:b_{\tau(i)}]$ by assumption, we see that $\sigma$ sends $(Z_i)_R$ to $(Z_{\tau(i)})_R$ by Corollary \ref{cor: image of singular locus under phi}. 
	Let $\cI_i$ be the ideal sheaf defining $(Z_i)_R$, then there is an idenfication
	\[ 
		H^0(\bP_R^{2M+1}, \cI_i(1))=\Ker (H^0(\bP_R^{2M+1}, \cO(1)) \rightarrow H^0((Z_i)_R,\cO(1)))=\pencil{x_0,y_0,\dots, \hat x_i, \hat y_i,\dots,x_M,y_M},
	\] and $\sigma^*$ sends $H^0(\bP_R^{2M+1}, \cI_{\tau(i)}(1))$ to $H^0(\bP_R^{2M+1}, \cI_i(1))$. 
	Now that $\pencil{x_{\tau(i)},y_\tau(i)}$ is sent to $\pencil{x_i,y_i}$ follows from the identities $\pencil{x_\tau(i),y_\tau(i)}=\bigcap_{\tau(j)\neq \tau(i)} H^0(\bP_R^{2M+1}, \cI_{\tau(j)}(1))$ and $\pencil{x_i,y_i}=\bigcap_{j\neq i} H^0(\bP_R^{2M+1}, \cI_i(1))$. 
	So we get $\sigma^*(x_{\tau(i)}+y_{\tau(i)})=ax_i+by_i$, hence $\sigma^*(x_{\tau(i)}^2+y_{\tau(i)}^2)=a^2x_i^2+b^2y_i^2$. 
	But a similar argument on $H^0(\bP_R^{2M+1}, \cI_i(2))$ shows that $\pencil{x_{\tau(i)}^2+y_{\tau(i)}^2}$ is sent to $\pencil{x_i^2+y_i^2}$. 
	Hence $a=b$ and $\pencil{x_{\tau(i)}+y_{\tau(i)}}$ is sent to $\pencil{x_i+y_i}$. 
	So any action $\sigma^*$ on $H^0(\bP^{2M+1},\cO(1))$ induced by $\sigma\in \Aut_{X,\cO_X(1)}^{\bG_m}(k)$ has the form
	\begin{align*}
		x_{\tau(i)} &\mapsto \alpha_i x_i + \beta_i y_i,\\
		y_{\tau(i)} &\mapsto \gamma_i x_i + \xi_i y_i
	\end{align*}
	such that $\alpha_i+\gamma_i=\beta_i+\xi_i$. 
	The condition $\alpha_i+\gamma_i=\beta_i+\xi_i$ is equivalent to $\alpha_i+\beta_i=\gamma_i+\xi_i$, and we denote this sum with $c$.  
	We compute the action of $\sigma^*$ on $f$ explicity:
	\begin{align*}
		\sigma^*f 
			&= \sum_{i=0}^M a_{\tau(i)}(\alpha_i x_i + \beta_i y_i)
								(\gamma_i x_i + \xi_i y_i)= \sum_{i=0}^M a_{\tau(i)}(\alpha_i\xi_i+ \beta_i\gamma_i)x_iy_i + 
							a_{\tau(i)}\alpha_i\gamma_i x_i^2 + a_{\tau(i)}\beta_i\xi_i y_i^2.
	\end{align*}
	Since $\sigma^* f$ still lies in the pencil $\pencil{f,g}$, so the coefficients of $x_i^2$ and $y_i^2$ must agree, which means
	\[
		\alpha_i\gamma_i=\beta_i\xi_i
	\]
	as $a_{\tau(i)}\neq 0$. 
	Replacing $\gamma_i$ with $c-\xi_i$ and $\beta_i$ with $c-\alpha_i$ yields
	\[
		c\alpha_i-\alpha_i\xi_i=c\xi_i-\alpha_i\xi_i.
	\]
	If $c=0$, then $\alpha_i=\beta_i$, and hence $\gamma_i=\xi_i$. 
	This is impossible since $\sigma^*\pencil{x_{\tau(i)},y_{\tau(i)}}$ should have dimension $2$.
	So we get $c\neq 0$, then $\alpha_i=\xi_i$, and hence $\beta_i=\gamma_i$.
\end{proof}

\section{Study of $\Aut^0$}
Let $k$ be an algebraically closed field of characteristic $2$, and let $\pencil{f,g}$ be a pencil of quadratic forms on $\bP^{2M+1}_k$. 
Define $X:=V(f,g)\subset \bP^{2M+1}_k$. 
We assume throughout this section that $X$ is smooth and the Pfaffian of $\pencil{f,g}$ has distinct roots.
One should think of $X$ as the base locus of $\pencil{f,g}$. 
We study in this section the structure of the connected component of the identity of the automorphism scheme $\Aut_X$. 
The main result of this section is

\begin{thm}\label{thm:structure of Aut^0}
	The connected component of identity $\Aut^0_X$ in $\Aut_X$ is isomorphic to $\mu_2^M$, if $M\geq 2$. 
\end{thm}

Note that by Corollary \ref{cor:Aut is finite}, the connected component $\Aut^0_X$ is the spectrum of a local Artinian algebra, so to compute $\Aut^0_X$, it suffices to compute the $A$-valued points $\Aut^0_X(A)$ for all local Artinian $k$-algebras $A$.

\begin{lem}\label{lem:pushforward of Aut^0 is trivial}
	If $M\geq 2$, then $\varphi_*(\Aut^0_X)$ is trivial.
\end{lem}
\begin{proof}
	By Proposition \ref{prop: pushforawrd of Aut sends image of singular locus to itself}, the pushforward $\varphi_*(\Aut^0_X)$ sends the image of the singular locus $\varphi(Z)$ of $X$ to itself. 
	Now wlog we take $\pencil{f,g}$ to be the normal form (Corollary \ref{cor:normal form of pencil}):
	\[
			f= \displaystyle\sum_{i=0}^M a_ix_iy_i,\ 
			g= \displaystyle\sum_{i=0}^M b_ix_iy_i+x_i^2+y_i^2,
 	\]
	where $\det \left[\begin{smallmatrix}
		a_i & a_j\\
		b_i & b_j
	\end{smallmatrix}\right] \neq 0$ for all $i\neq j$. 
	By Corollary \ref{cor: image of singular locus under phi}, the image $\varphi(Z)$ is the distinct $M+1$ points $\coprod_{i=0}^M [a_i:b_i]$. 
	Since $\varphi_*(\Aut_X^0)$ is a connected subgroup of $\PGL_2$, it operates trivially on $\varphi(Z)$, and hence $\varphi_*(\Aut_X^0) \subset \Stab_{\PGL_2}(\varphi(Z))$, where the latter is trivial for $M\geq 2$. 
\end{proof}

\begin{proof}[Proof of Theorem \ref{thm:structure of Aut^0}]
	Write $[x_0:\dots:x_M:y_0:\dots:y_M]$ for the coordinates in $\bP^{2M+1}_k$ and take $\pencil{f,g}$ in the normal form given in Corollary \ref{cor:normal form of pencil}: 
	\[
			f= \displaystyle\sum_{i=0}^M a_ix_iy_i,\ 
			g= \displaystyle\sum_{i=0}^M b_ix_iy_i+x_i^2+y_i^2,
 	\]
	where $a_i\neq 0$ for all $i$, and $\det \left[\begin{smallmatrix}
		a_i & a_j\\
		b_i & b_j
	\end{smallmatrix}\right] \neq 0$ for all $i\neq j$. 
	Let $(A,\fm)$ be a local Artinian $k$-algebra, and let $\sigma \in \left(\Aut^{\bG_m}_{X,\cO_X(1)}\right)^0(A) \subset \GL_{2M+2}(A)$ be an element of the lifting group of $\Aut^0_{X,\cO_X(1)}$. 
	By Proposition \ref{prop:action of elements in pi_0 on subspace} and Lemma \ref{lem:pushforward of Aut^0 is trivial}, we may write
	\begin{align*}
		\sigma^*(x_i) =\alpha_i x_i + \beta_i y_i,\\
		\sigma^*(y_i) =\beta_i x_i + \alpha_i y_i.
	\end{align*}
	Since the support of the image of $\Spec A \rightarrow \left(\Aut^{\bG_m}_{X,\cO_X(1)}\right)^0$ induced by $\sigma$ is the identity, we may assume $\alpha_i-1\in \fm, \beta_i \in \fm$ for all $i$. 
	In particular, $\alpha_i$ are invertible.

	We compute the action of $\sigma$ on $\pencil{f,g}$:
		\begin{align*}
	 		\sigma^*f &= \sum_{i=0}^M a_i (\alpha_ix_i+\beta_iy_i)(\beta_ix_i+\alpha_iy_i)\\
	 		&= \sum_{i=0}^M a_i(\alpha_i^2+\beta_i^2) x_iy_i + a_i\alpha_i\beta_i (x_i^2+y_i^2),\\
	 		\sigma^*g &= \sum_{i=0}^M b_i (\alpha_ix_i+\beta_iy_i)(\beta_ix_i+\alpha_iy_i)
	 				+ (\alpha_ix_i+\beta_iy_i)^2 + (\beta_ix_i+\alpha_iy_i)^2\\
	 		&= \sum_{i=0}^M b_i(\alpha_i^2+\beta_i^2) x_iy_i + (b_i\alpha_i\beta_i+\alpha_i^2+\beta_i^2) (x_i^2+y_i^2).
	 	\end{align*}
	By Lemma \ref{lem:pushforward of Aut^0 is trivial}, $\sigma$ acts on the pencil $\pencil{f,g}$ by a scalar, i.e. $\sigma^*f=cf, \sigma^*g=cg$ for some $c\in A^\times$. 
	By comparing coefficients we get the following equations:
	\[
		\alpha_i^2+\beta_i^2=c,\ \alpha_i\beta_i=0
	\]	
	for all $j$ and all $j\neq k$.
	Since $\alpha_i$ are invertible, we get $\beta_i=0$ and $\alpha_i^2=c$ for all $i$.
	This shows
	\[
		\left(\Aut^{\bG_m}_{X,\cO_X(1)}\right)^0(A)=\left\{ 
			\mathrm{diag}(\alpha_0,\dots, \alpha_M)\in \GL_{2M+2}(A)
			\ \vline\ 
			\alpha_{ii}^2=c \ \ \forall\ 0\leq i \leq M
		\right\}.
	\]
	Then by Proposition \ref{prop: exact sequence of lifting group}, $\Aut_{X,\cO_X(1)}(A)=\left(\Aut^{\bG_m}_{X,\cO_X(1)}\right)^0(A) /\bG_m(A)$ and one sees directly that $\Aut_{X,\cO_X(1)}=\mu_2^M$. 
	We conclude using Proposition \ref{prop:Aut and polarized Aut agree}.
\end{proof}

\section{Study of $\pi_0(\Aut_X)$}
Let $k$ be an algebraically closed field of characteristic $2$, and let $\pencil{f,g}$ be a pencil of quadratic forms on $\bP^{2M+1}_k$. 
Define $X:=V(f,g)\subset \bP^{2M+1}_k$. 
We assume that $X$ is smooth and the Pfaffian of $\pencil{f,g}$ has distinct roots throughout this section. 
We study in this section the structure of the reduced group scheme $(\Aut_X)_{\text{red}}=\pi_0(\Aut_X)$. 
Using the isomorphism $\Aut_X=\Aut^0_X\rtimes \pi_0(\Aut_X)$, we identify $\pi_0(\Aut_X)$ as a subgroup of $\Aut_X$. 
Note also $\pi_0(\Aut_X)=\underline{\Aut_X(k)}$.
Therefore, it suffices to consider automorphisms of $X$ only (but not $X_T$ for some base change along $T$).

By Lemma \ref{lem:pushforward of Aut^0 is trivial}, $\varphi_*(\Aut^0_X)$ is trivial, so $\varphi_*$ factors through $\Aut_X/\Aut^0_X=\pi_0(\Aut_X)$. We denote the factorization $\pi_0(\Aut_X) \rightarrow \PGL_2$ also by $\varphi_*$ by an abuse of notation. 
The main results of this section are

\begin{thm}\label{thm: exact sequence for pi_0}
	Let $\varphi$ be the morphism $\bP_k^{2M+1}\backslash X \rightarrow \bP_k^1$ defined by $\pencil{f,g}$, and let $Z$ be the singular locus of $\varphi$ defined in Proposition \ref{prop:locus of fitting ideal is singular locus}. 
	If $M\geq 2$, then there is an exact sequence of groups:
	\[\begin{tikzcd}
		1 & {(\bZ/2\bZ)^{M+1}} & {\pi_0(\Aut_X)} & {\Aut(\bP_k^1;\varphi(Z)),}
		\arrow[from=1-1, to=1-2]
		\arrow[from=1-2, to=1-3]
		\arrow["\varphi_*",from=1-3, to=1-4]
	\end{tikzcd}\]
	where $\Aut(\bP_k^1;\varphi(Z))$ is the equivariant automorphism group of $\varphi(Z)$ in $\bP_k^1$, that is, the subgroup of $\PGL_2(k)$ whose action on $\bP_k^1$ sends $\varphi(Z)$ to itself (but possibly permutes the points).
\end{thm}
\begin{note}
	Assume that $\pencil{f,g}$ is given by the normal form 
	\[
		f= \displaystyle\sum_{i=0}^M a_ix_iy_i,\ 
		g= \displaystyle\sum_{i=0}^M b_ix_iy_i+x_i^2+y_i^2,
 	\]
	where $a_i\neq 0$ for all $i$, and $\det \left[\begin{smallmatrix}
		a_i & a_j\\
		b_i & b_j
	\end{smallmatrix}\right] \neq 0$ for all $i\neq j$ by Corollary \ref{cor:normal form of pencil}. 
	Recall from Notation \ref{not: notation for the scalars} we may associate a permutation $\tau\in S_{M+1}$ to every $\zeta\in \Aut(\bP_k^1;\varphi(Z))$, by its action on $\varphi(Z)$.
	Choose a matrix representation $A$ of $\zeta$, which is equivalent to choose a preimage of $\zeta \in \PGL_2(k)$ under the quotient $\GL_2 \rightarrow \PGL_2$, there exist $\lambda_i\in k^\times$ such that
	\[
		A\cdot \begin{bmatrix}
			a_i\\
			b_i
		\end{bmatrix}
		= \lambda_i \cdot \begin{bmatrix}
			a_{\tau(i)}\\
			b_{\tau(i)}
		\end{bmatrix}.
	\]
\end{note}
\begin{thm}\label{thm:characterization of image of phi_*}
	The automorphism $\zeta$ is in the image of $\varphi_*$ if and only if the following two conditions hold:
	\begin{enumerate}
		\item $\det A= \frac{a_{\tau(i)}\lambda_i^2}{a_i}$ for all $0\leq i \leq M$,
		\item For each $0\leq j\leq M$, the value 
		\[
			\frac{a_{\tau(i)}(\lambda_i-\lambda_j)}{a_{\tau(i)}b_{\tau(j)}-a_{\tau(j)}b_{\tau(i)}}
		\]
		is independent of $i$, for $i\neq j$.
	\end{enumerate}
	 
\end{thm}

\begin{rmk}
	We remark here that the values $\det A$, $\frac{a_{\tau(i)}\lambda_i^2}{a_i}$ and $\frac{a_{\tau(i)}(\lambda_i-\lambda_j)}{a_{\tau(i)}b_{\tau(j)}-a_{\tau(j)}b_{\tau(i)}}$ are dependent of the choice of a matrix representative of $A$, but the two conditions are independent of the choice of the representative. 
	However, they indeed depend on the choice of the representative of the points $[a_i:b_i]$. 
	The values $a_i$ and $b_i$ here should be considered as given canonically by the normal form, and cannot be replaced by scalars.
\end{rmk}

\begin{proof}[Proof of Theorem \ref{thm: exact sequence for pi_0}]
	By Proposition \ref{prop: pushforawrd of Aut sends image of singular locus to itself}, the pushforward $\varphi_*(\Aut_X)$ sends $\varphi(Z)$ of $X$ to itself, so the image of $\pi_0(\Aut_X)$ under $\varphi_*$ lies in $\Aut(\bP_k^1;\varphi(Z))$. 
	
	We still need to prove that $\ker \varphi_*=(\bZ/2\bZ)^{M+1}$. 
	Consider the following diagram with exact rows:
	\[\begin{tikzcd}
		1 & 1 & G & {\ker \varphi_*} & 1 \\
		1 & {k^*} & {\Aut_{X,\cO_X(1)}^{\bG_m}(k)} & {\Aut_X(k)} & 1 \\
		1 & {k^*} & {\GL_2(k)} & {\PGL_2(k)} & 1,
		\arrow[from=1-1, to=1-2]
		\arrow[from=1-2, to=1-3]
		\arrow[from=1-2, to=2-2]
		\arrow[from=1-3, to=1-4]
		\arrow[from=1-3, to=2-3]
		\arrow[from=1-4, to=1-5]
		\arrow[from=1-4, to=2-4]
		\arrow[from=2-1, to=2-2]
		\arrow[from=2-2, to=2-3]
		\arrow["{(-)^2}", from=2-2, to=3-2]
		\arrow[from=2-3, to=2-4]
		\arrow["{\varphi_*}", from=2-3, to=3-3]
		\arrow[from=2-4, to=2-5]
		\arrow["{\varphi_*}", from=2-4, to=3-4]
		\arrow[from=3-1, to=3-2]
		\arrow[from=3-2, to=3-3]
		\arrow[from=3-3, to=3-4]
		\arrow[from=3-4, to=3-5]
	\end{tikzcd}\]
	where the groups in the top row are taken as the kernel of the vertical arrows below. 
	By a diagram chase similar to the proof of snake lemma, it can be easily shown that the top row is exact. 
	Hence it suffices to show that $G=(\bZ/2\bZ)^{M+1}$. 
	Pick any $\sigma \in G$. 
	By Proposition \ref{prop:action of elements in pi_0 on subspace}, any action $\sigma^*$ on $H^0(\bP^{2M+1},\cO(1))$ induced by $\sigma\in G$ has the form
	\begin{align*}
		x_i &\mapsto \alpha_i x_i + \beta_i y_i,\\
		y_i &\mapsto \beta_i x_i + \alpha_i y_i.
	\end{align*}
	We compute the action of $\sigma$ on $f$ explicitly:
	\begin{align*}
		\sigma^*f 
			= \sum_{i=0}^M a_i(\alpha_i x_i + \beta_i y_i)
								(\beta_i x_i + \alpha_i y_i)
		= \sum_{i=0}^M a_i(\alpha_i^2+ \beta_i^2)x_iy_i + 
							a_i\alpha_i\beta_i x_i^2 + a_i\alpha_i\beta_i y_i^2.
	\end{align*}
	 If $\sigma\in G$, then the induced action of $\sigma^*$ on the linear system $\pencil{f,g}\subset H^0(\bP^{2M+1}_k,\cO(2))$ is trivial, so $\sigma^*f=f$. 
	 Comparing coefficients, we get
	 \[
	 	\alpha_i^2+\beta_i^2=1 \ ,\ \alpha_i\beta_i=0
	 \]
	 for all $0\leq i \leq M$. 
	 We see then that the only two possibilities are
	 \[\begin{array}{cc}
	 	\begin{cases}
	 		\alpha_i=1\\
	 		\beta_i=0
	 	\end{cases}, &
	 	\begin{cases}
	 		\alpha_i=0\\
	 		\beta_i=1
	 	\end{cases}.
	 \end{array}\]
	 Therefore, $\sigma^*$ either acts trivially on $\pencil{x_i,y_i}$ or swaps $x_i$ and $y_i$, hence $G\cong (\bZ/2\bZ)^{M+1}$.
\end{proof}

\begin{proof}[Proof of Theorem \ref{thm:characterization of image of phi_*}]
	Let $\zeta$ be an automorphism in $\Aut(\bP^1_k ; \varphi(Z))$. 
	Let $\tau\in S_{M+1}$ be the permutation associated to $\zeta$ according to Notation \ref{not: notation for the scalars}. 
	Consider the following diagram with exact rows: 
	\[\begin{tikzcd}
		1 & {k^*} & {\Aut_{X,\cO_X(1)}^{\bG_m}(k)} & {\Aut_X(k)} & 1 \\
		1 & {k^*} & {\GL_2(k)} & {\PGL_2(k)} & 1.
		\arrow[from=1-1, to=1-2]
		\arrow[from=1-2, to=1-3]
		\arrow["{(-)^2}", from=1-2, to=2-2]
		\arrow[from=1-3, to=1-4]
		\arrow["{\varphi_*}", from=1-3, to=2-3]
		\arrow[from=1-4, to=1-5]
		\arrow["{\varphi_*}", from=1-4, to=2-4]
		\arrow[from=2-1, to=2-2]
		\arrow[from=2-2, to=2-3]
		\arrow[from=2-3, to=2-4]
		\arrow[from=2-4, to=2-5]
	\end{tikzcd}\]
	We prove first the forward implication. 
	Let $\sigma_0 \in \Aut_X(k)$ be a preimage of $\zeta$, and let $\sigma \in \Aut_{X,\cO_X(1)}^{\bG_m}(k)$ be a lift of $\sigma_0$. 
	By Proposition \ref{prop:action of elements in pi_0 on subspace}, the induced action $\sigma^*$ on $H^0(\bP^{2M+1}_k, \cO(1))$ has the following form:
	\begin{align*}
		x_{\tau(i)} &\mapsto \alpha_i x_i + \beta_i y_i,\\
		y_{\tau(i)} &\mapsto \beta_i x_i + \alpha_i y_i.
	\end{align*}
	By Corollary \ref{cor: image of singular locus under phi}, the singular elements of the pencil $\pencil{f,g}$ are the fibres of $\varphi_{\pencil{f,g}}$ over each $[a_i:b_i]$. We denote the singular elements in $\pencil{f,g}$ corresponding to $[a_i:b_i]$ with $h_i$. Concretely, we have
	\[
		h_i=a_ig-b_if=\sum_{j=0}^M(a_ib_j-a_jb_i) x_jy_j + a_ix_j^2+a_iy_j^2.
	\]
	Since $\varphi_*(\sigma)$ sends $[a_i:b_i]$ to $[a_{\tau(i)}:b_{\tau(i)}]$, we must have $\sigma^*(h_{\tau(i)})=\mu_i h_i$ for some scalar $\mu_i\neq 0$.
	On the other side, we can compute the action of $\sigma^*$ on $h_{\tau(i)}$ explicitly:
	\begin{align*}
		\sigma^*(h_{\tau(i)}) &=
		\sum_{\tau(j)=0}^M(a_{\tau(i)}b_{\tau(j)}-a_{\tau(j)}b_{\tau(i)})(\alpha_j x_j+\beta_j y_j)(\beta_j x_j+\alpha_j y_j)
		+ a_{\tau(i)}(\alpha_j^2 x_j^2+\beta_j^2 y_j^2)+
		a_{\tau(i)}(\beta_j^2 x_j^2+\alpha_j^2 x_j^2)\\
		&=\sum_{\tau(j)=0}^M(a_{\tau(i)}b_{\tau(j)}-a_{\tau(j)}b_{\tau(i)})(\alpha_j^2+\beta_j^2)x_jy_j +
		\left((a_{\tau(i)}b_{\tau(j)}-a_{\tau(j)}b_{\tau(i)})\alpha_j\beta_j + a_{\tau(i)}(\alpha_j^2+\beta_j^2)\right)x_j^2 
		\\&+ 
		\left((a_{\tau(i)}b_{\tau(j)}-a_{\tau(j)}b_{\tau(i)})\alpha_j\beta_j + a_{\tau(i)}(\alpha_j^2+\beta_j^2)\right)y_j^2.
	\end{align*}
	Comparing the coefficient of $x_i^2$ with the one in $\mu_ih_i$, we get
	\[
		a_{\tau(i)}(\alpha_i^2+\beta_i^2)=\mu_i a_i.
	\]
	Comparing the coefficients of $x_jy_j$ with the ones in $\mu_ih_i$, we get
	\begin{align*}
		(a_{\tau(i)}b_{\tau(j)}-a_{\tau(j)}b_{\tau(i)})(\alpha_j^2+\beta_j^2)
		&=\mu_i(a_ib_j-a_jb_i).
	\end{align*}
	Let $A$ be the matrix form of $\varphi_*(\sigma) \in \GL_2$ with basis $f,g$, and define $\lambda_i$ as in Notation \ref{not: notation for the scalars}.
	We note that
	\[
		A\cdot \begin{bmatrix}
			a_i & a_j\\
			b_i & b_j
		\end{bmatrix}
		=\begin{bmatrix}
			\lambda_ia_{\tau(i)} & \lambda_ja_{\tau(j)}\\
			\lambda_ib_{\tau(i)} & \lambda_jb_{\tau(j)}
		\end{bmatrix},
	\]
	which implies $\det A \cdot (a_ib_j-a_jb_i)=\lambda_i\lambda_j (a_{\tau(i)}b_{\tau(j)}-a_{\tau(j)}b_{\tau(i)})$.
	So the equation above reads also
	\[
		\frac{\alpha_j^2+\beta_j^2}{\lambda_j}=\frac{\lambda_i \mu_i}{\det A}.
	\]
	Since $\frac{\alpha_j^2+\beta_j^2}{\lambda_j}$ is independent of the choice of the index $i$, we must have $\frac{\lambda_i \mu_i}{\det A}$ is independent of $i$. 
	We denote this constant by $c$
	In particular, we have also
	\[
		\alpha_i^2+\beta_i^2=c\lambda_i=\frac{\mu_i\lambda_i^2}{\det A}.
	\]
	Combining with $a_{\tau(i)}(\alpha_i^2+\beta_i^2)=\mu_i a_i$, we get
	\[
		\det A= \frac{a_{\tau(i)}\lambda_i^2}{a_i}.
	\]
	This shows the first condition. 
	For the second condition, we compare the coefficients of $x_j^2$ in $\sigma^*(h_{\tau(i)})$ and $\mu_i h_i$, which reads
	\[
		(a_{\tau(i)}b_{\tau(j)}-a_{\tau(j)}b_{\tau(i)})\alpha_j\beta_j + a_{\tau(i)}(\alpha_j^2+\beta_j^2)=\mu_i a_i.
	\]
	Replacing $\alpha_j^2+\beta_j^2$ by $\frac{\lambda_i\lambda_j\mu_i}{\det A}$, and $a_i$ by $\frac{a_{\tau(i)}\lambda_i^2}{\det A}$, we obtain
	\[
		(a_{\tau(i)}b_{\tau(j)}-a_{\tau(j)}b_{\tau(i)})\alpha_j\beta_j=
		\frac{\mu_ia_{\tau(i)}\lambda_i^2-\mu_ia_{\tau(i)}\lambda_i\lambda_j}{\det A}.
	\]
	Replacing $\frac{\mu_i\lambda_i}{\det A}$ by $c$, we see
	\[
		\alpha_j\beta_j=c\cdot \frac{a_{\tau(i)}(\lambda_i-\lambda_j)}{a_{\tau(i)}b_{\tau(j)}-a_{\tau(j)}b_{\tau(i)}}.
	\]
	Hence the right hand side is independent of $i$, showing the second condition.
	
	For the backward implication, we pick a matrix representative $A$ of $\zeta$ (or equivalently a preimage of $\zeta$ under $\GL_2 \rightarrow \PGL_2$), and $\lambda_i\in k^\times$ such that for each $i$,
	\[
		A\cdot \begin{bmatrix}
			a_i\\
			b_i
		\end{bmatrix}
		= \lambda_i \cdot \begin{bmatrix}
			a_{\tau(i)}\\
			b_{\tau(i)}
		\end{bmatrix}
	\]
	for some scalar $\lambda_i\neq 0$. 
	Let $\alpha_j,\beta_j$ be a solution of the following equation systems
	\[\begin{cases}
		\alpha_j+\beta_j &= \sqrt{\lambda_j},\\
		\alpha_j\beta_j  &= \frac{a_{\tau(i)}(\lambda_i-\lambda_j)}{a_{\tau(i)}b_{\tau(j)}-a_{\tau(j)}b_{\tau(i)}}.
	\end{cases}\]
	Consider the automorphism $\sigma\in\Aut_{X,\cO_X(1)}^{\bG_m}(k)$ defined by
	\begin{align*}
		x_{\tau(i)}&\mapsto \alpha_i x_i+\beta_i y_i,\\
		y_{\tau(i)}&\mapsto \beta_i x_i+\alpha_i y_i.
	\end{align*}
	We may compute its action on $h_{\tau(i)}$:
	\begin{align*}
		\sigma^*(h_{\tau(i)})&=
		\sum_{\tau(j)=0}^M(a_{\tau(i)}b_{\tau(j)}-a_{\tau(j)}b_{\tau(i)})\lambda_j x_jy_j + a_{\tau(i)}\lambda_i \cdot (x_j^2+y_j^2)\\
		&=\sum_{\tau(j)=0}^M \frac{\det A \cdot (a_ib_j-a_jb_i)}{\lambda_i} x_jy_j + a_{\tau(i)}\lambda_i \cdot (x_j^2+y_j^2) =\frac{a_{\tau(i)}\lambda_i}{a_i} \cdot h_i
	\end{align*}
	where the second equality is due to the fact $\det A \cdot (a_ib_j-a_jb_i)=\lambda_i\lambda_j (a_{\tau(i)}b_{\tau(j)}-a_{\tau(j)}b_{\tau(i)})$ proven in the forward implication. 
	In particular, $\sigma$ fixes the pencil $\pencil{f,g}$ since any two $h_i,h_j$ build a basis of $\pencil{f,g}$, and its pushforward $\varphi_*(\sigma)$ sends $[a_i: b_i]$ to $[a_{\tau(i)}: b_{\tau(i)}]$. 
	So $\varphi_*(\sigma)$ agrees with $A$ up to a scalar as $\PGL_2$ is $3$-transitive.
\end{proof}

\begin{cor}\label{cor: pi_0 for a generic X}
	For a generic smooth $(2,2)$-complete intersection $X$, we have $\pi_0(\Aut_X) \cong (\bZ/2\bZ)^{M+1}$.
\end{cor}

\begin{rmk}
	It is interesting to compare Theorem \ref{thm:characterization of image of phi_*} with the case where the complete intersection is even-dimensional or the cases of other characteristics.  
	Indeed, it is well known for characteristic $\neq 2$ that $\varphi_*$ is surjective, and the surjectivity for the case of even dimension in characteristic $2$ is shown by \cite[Theorem 1.6]{dolgachev2018regular}. 
	In our case, the two conditions appear to be very restrictive. 
	The following first example gives a non-trivial liftable element in $\Aut(\bP_k^1 ; \varphi(Z))$ with respect to a certain pencil, and the second example shows that $\varphi_*$ is not always surjective.
\end{rmk}

\begin{ex}\label{ex: liftable element in PGL_2}
	Let $a$ be a primitive $(M+1)$-th root of unity. We consider the pencil
	\begin{align*}
		f=\sum_{i=0}^M a^i \cdot x_iy_i,\ 
		g=\sum_{i=0}^M x_iy_i+x_i^2+y_i^2
	\end{align*}
	in $H^0(\bP^{2M+1}_k, \cO(2))$. The singular elements in the pencil are the fibres over $[a^i: 1]$ for $0\leq a \leq M-1$. 
	We take $A=\begin{bmatrix} a & 0\\0 & 1\end{bmatrix}$, and we see that
	\[
		\begin{bmatrix}
			a & 0\\
			0 & 1
		\end{bmatrix}\cdot
		\begin{bmatrix}
			a^i\\ 1
		\end{bmatrix}=
		\begin{bmatrix}
			a^{i+1}\\ 1
		\end{bmatrix},
	\]
	so $\lambda_i=1$ for all $i$. 
	This configuration satisfies the conditions in Theorem \ref{thm:characterization of image of phi_*}, and $A$ can be lifted to the automorphism
	\begin{alignat*}{3}
		x_0 &\mapsto x_M,\ &y_0 &\mapsto y_M, \\
		x_i &\mapsto x_{i-1},\ &y_i &\mapsto y_{i-1},\ \forall 1\leq i \leq M.
	\end{alignat*}
\end{ex}

\begin{ex}\label{ex: non-liftable element in PGL_2}
	Let $a$ be a root of $x^2+x+1$ in $k$. 
	We consider the pencil of quadratic forms
	\begin{alignat*}{3}
		f &= &x_0y_0+&x_1y_1+x_2y_2+ax_3y_3,\\
		g &= &		&x_1y_1+ax_2y_2+x_3y_3+x_0^2+x_1^2+x_2^2+x_3^2+y_0^2+y_1^2+y_2^2+y_3^2
	\end{alignat*}
	in $H^0(\bP_k^7,\cO(2))$. 
	The singular elements in the pencil are the fibres over $[1:0],[1:1],[1:a]$ and $[a:1]$ along $\varphi_{\pencil{f,g}}$. 
	We take $A=\begin{bmatrix} 1 & 0\\0 & a\end{bmatrix}\in \GL_2(k)$ and compute
	\[
		\begin{bmatrix}
			1 & 0\\
			0 & a
		\end{bmatrix}\cdot
		\begin{bmatrix}
			1\\ 0
		\end{bmatrix}=
		\begin{bmatrix}
			1\\ 0
		\end{bmatrix}\ ,\ 
		\begin{bmatrix}
			1 & 0\\
			0 & a
		\end{bmatrix}\cdot
		\begin{bmatrix}
			1\\ 1
		\end{bmatrix}=
		\begin{bmatrix}
			1\\ a
		\end{bmatrix}\ ,\ 
		\begin{bmatrix}
			1 & 0\\
			0 & a
		\end{bmatrix}\cdot
		\begin{bmatrix}
			1\\ a
		\end{bmatrix}= (1+a)\cdot
		\begin{bmatrix}
			a\\ 1
		\end{bmatrix}\ ,\ 
		\begin{bmatrix}
			1 & 0\\
			0 & a
		\end{bmatrix}\cdot
		\begin{bmatrix}
			a\\ 1
		\end{bmatrix}=a\cdot
		\begin{bmatrix}
			1\\ 1
		\end{bmatrix}.
	\]
	So $\lambda_0=\lambda_1=1,\lambda_2=1+a,\lambda_3=a$ following Notation \ref{not: notation for the scalars}. And
	\[
		\det A =a \neq 1= \frac{a_{\tau(0)}\lambda_0^2}{a_0}.
	\]
	By Theorem \ref{thm:characterization of image of phi_*}, the matrix $\begin{bmatrix} 1 & 0\\0 & a\end{bmatrix}$ cannot be lifted to an automorphism of $V(f,g)$. 
\end{ex}

\bibliographystyle{alpha}
\bibliography{biblio}

\end{document}